\documentclass{article}

\usepackage{amsmath}
\usepackage{amssymb}
\usepackage{graphicx}
\usepackage{epic,eepic}

\usepackage{color}

\usepackage{amstext,amsthm,amssymb,amsmath}
\usepackage{epsfig,subfigure,epstopdf}
\usepackage{graphicx}
\usepackage{subfigure}
\usepackage{wrapfig}
\usepackage{fullpage}
\usepackage{color}
\usepackage{multirow}
\usepackage{tabulary}
\usepackage{booktabs}
\usepackage{enumerate}

\newtheorem{theorem}{Theorem}
\newtheorem{lemma}{Lemma}

\usepackage[utf8]{inputenc}
\usepackage[english]{babel}
\usepackage{hyperref}

\usepackage{nomencl}

\theoremstyle{definition}

\makeatother

\usepackage{babel}

\title{Multiscale stabilization for convection diffusion equations with heterogeneous velocity and diffusion coefficients}

\author{
Eric T. Chung \thanks{Department of Mathematics,
The Chinese University of Hong Kong (CUHK), Hong Kong SAR. Email: {\tt tschung@math.cuhk.edu.hk}.
The research of Eric Chung is supported by Hong Kong RGC General Research Fund (Project 14317516)
and CUHK Direct Grant for Research 2017-18.}
\and
Yalchin Efendiev \thanks{Department of Mathematics \& Institute for Scientific Computation (ISC),
Texas A\&M University,
College Station, Texas, USA. Email: {\tt efendiev@math.tamu.edu}.}
\and
Wing Tat Leung \thanks{The Center for Subsurface Modeling, The Institute for Computational Engineering and Sciences, The University of Texas at Austin, Austin, TX 78712}
}

\begin{document}

\maketitle
\begin{abstract}
We present a new stabilization technique for multiscale convection diffusion problems. Stabilization for these problems has been a challenging task,
especially for the case with high Peclet numbers. Our method is based on a constraint energy minimization idea and the discontinuous Petrov-Galerkin formulation. In particular, the test functions
are constructed by minimizing an appropriate energy subject to certain orthogonality conditions, and are related to the trial space.
The resulting test functions have a localization property,
and can therefore be computed locally. We will prove the stability, and present several numerical results.
Our numerical results confirm that our test space gives a good stability, in the sense that the solution error is close
to the best approximation error.
\end{abstract}

\section{Introduction}

In this paper, we consider a class of convection-diffusion problems in the form
\begin{equation}
\label{eq:pde}
-\nabla\cdot(\kappa\nabla u)+b\cdot\nabla u  =f
\end{equation}
with a high Peclet number, where $\kappa$ is a diffusion tensor and
$b$ is a velocity vector~\cite{park2004multiscale,
  fannjiang1994convection}.  We assume that both fields contain
multiscale spatial features with high contrast.  There are in literature a wide range of numerical schemes for this problem that are based on constructions of special basis functions on coarse
grids~\cite{dur91, weh02, Arbogast_two_scale_04, egw10, egh12, eh09,
  g1, g2, Review, Efendiev_GKiL_12, ehg04, Chu_Hou_MathComp_10,ee03,
  calo2011note, calo2014asymptotic, GhommemJCP2013, eglmsMSDG}.  These
methods include the Multiscale Finite Element Methods
(MsFEM)~\cite{egw10, egh12, eh09, ehg04, Ensemble, alotaibi2015global},
the Variational Multiscale Methods~\cite{hughes98,
  hughes1995multiscale, hsu2010improving, bazilevs2010isogeometric,
  masud2004multiscale, buffa2006analysis, hughes2005variational,
  codina1998comparison, bazilevs2007variational, akkerman2008role,
  hughes2007variational} and
the Generalized Multiscale Finite Element Method
(GMsFEM)
\cite{egh12, galvis2015generalized,Ensemble, eglmsMSDG,
  eglp13oversampling, calo2014multiscale, chung2014adaptive, chung2014adaptiveDG,
  randomized2014, chung2015generalizedperforated, chung2015residual,
  chung2015online}.
  When the above approaches are used to solve
  multiscale convection-dominated diffusion problems with a high Peclet number,
besides finding a reduced approximate solution space, one needs to stabilize the
system to avoid large errors~\cite{park2004multiscale}.
It is known that simplified stabilization techniques
do not suffice for
complex problems and one needs a systematic approach to generate the
necessary test spaces.

 In
this paper, we will derive and analyze a new stabilization technique,
which
combines recent developments in Constraint Energy Minimizing Generalized Multiscale Finite Element
Method (CEM-GMsFEM) \cite{chung2017constraint} and Discontinuous Petrov-Galerkin method
(e.g.,~\cite{demkowicz2014overview, niemi2011discontinuous,
  niemi2013automatically}).
  To motivate our method, we
  start with a stable
fine-scale finite element discretization that fully resolves all
scales of (\ref{eq:pde})
\begin{equation}
  \label{eq:discrete}
  Au = f.
\end{equation}
We will apply the discontinuous Petrov-Galerkin (DPG) techniques
following~\cite{demkowicz2013primal, chan2014robust,
  demkowicz2013robust, demkowicz2014overview, demkowicz2012class,
  zitelli2011class} to stabilize the system. In particular, we will rewrite the above system
 in a mixed framework using an auxiliary variable
as follows
\begin{align}
  \label{eq:mixed11}
  R w + A u &= f,\\
  A^T w\ \ \qquad&=0,   \label{eq:mixed12}
\end{align}
where the variable $w$ plays the role of the test function and the matrix $R$
is related to the norm in which we seek to achieve stability.  We
assume that the fine-scale system gives $w=0$, that is, it is
discretely stable.
The aim of this paper is to design a space for the variable $w$, given
a choice of the trial space for $u$.

Within the DPG framework, one can achieve stability by choosing
test functions $w$ with global support~\cite{barrett1984approximate,
  demkowicz1986adaptive}.
  The least squares approaches~\cite{bochev1998finite,
  hughes1989new, bochev2009least, fuchen16} can be used to
achieve stability in the natural norm.
We also note that a stabilization technique based on the variational multiscale method is presented in \cite{li2017error}.
  Our goal is to design procedures
for constructing test functions that are localizable and give good stability, that work well for large Peclet numbers.
To construct our test functions, we assume that a given set of local multiscale trial functions is available,
and that they satisfy a stable decomposition property. We note that these functions can be constructed by,
for example, using the GMsFEM approach (see e.g.~\cite{egh12}).
To find the test functions, we use the idea of CEM-GMsFEM \cite{chung2017constraint}.
First, we will construct an auxiliary space.
In particular, for each coarse cell, we solve a spectral problem, which is defined
based on the above mixed formulation (\ref{eq:mixed12}). The first few eigenfunctions contain
important features about the multiscale coefficients $\kappa$ and $b$, and are used in
the construction of our test functions.
Using these local eigenfunctions and the given local trial functions, we will find the required test functions
by minimizing an appropriate energy subject to some constraints.
We will show that the test functions are localizable
and that they give good stability of the resulting numerical scheme.

We will present some numerical results to show the performance.
We will show the performance of using various coarse grid sizes and various choices of oversampling layers.
We observe that, once a sufficient number of oversampling layers is used, the solution error is very close to the projection error,
which is the best approximation error in the trial space. This confirms that our test space provides
a good stability even for high Peclet number.

The paper is organized as follows.
In Section \ref{sec:prelim}, we give some basic notations and the formulation of our problem.
In Section \ref{sec:stable}, we present the construction of the test space,
and show that the space gives a good stability.
Numerical results to validate the theory will be presented in Section \ref{sec:numer}.
Finally, a conclusion is given.

\section{Preliminaries}
\label{sec:prelim}

We will give some basic notations in this section.
In this paper, we consider convection diffusion problems of the form
\begin{equation}
-\mbox{div}\big(\kappa(x)\,\nabla u\big)+b(x)\cdot\nabla u=f\quad\text{in}\quad\Omega,\label{eq:original}
\end{equation}
subject to the homogeneous Dirichlet boundary condition $u=0$ on
$\partial\Omega$, where $\Omega \subset \mathbb{R}^d$ is the computational domain and $f\in L^2(\Omega)$ is a given source.
We
assume that both $\kappa(x)$ and $b(x)$ are heterogeneous coefficients with multiple
scales and very high contrast, and that $b(x)$ is divergence free. We assume $\kappa_0 \leq \kappa \leq \kappa_1$
where $\kappa_1/\kappa_0$ is large.

We next introduce the notion of fine and
coarse grids. We let $\mathcal{T}^{H}$ be a usual conforming partition
of the computational domain $\Omega$ into $N$ finite elements (triangles,
quadrilaterals, tetrahedra, etc.). We refer to this partition as the
coarse grid and assume that each coarse element is partitioned into
a connected union of fine grid blocks. The fine grid partition will
be denoted by $\mathcal{T}^{h}$, and we assume that this
is a refinement
of the coarse grid $\mathcal{T}^{H}$. See Figure~\ref{fig:illustration} for an illustration.
We let $N$ be the number of coarse elements and $N_c$ be the number of coarse grid nodes.
We remark that our test functions are defined with respect to the coarse grid, and the fine grid is used to compute
the test functions numerically.

\begin{figure}[ht!]
\centering
\includegraphics[width=2.5in, height=2.5in]{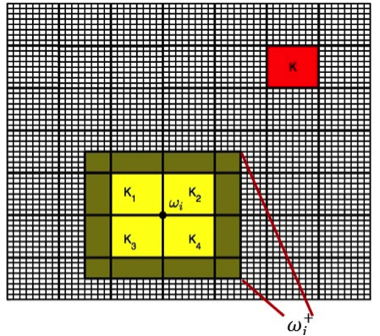}
\caption{Illustration of the coarse grid, fine grid, coarse neighborhood and  oversampling domain.}
\label{fig:illustration}
\end{figure}

Next, we give the precise formulation of our problem.
We define $V=H_{0}^{1}(\Omega)$
and
the bilinear form $a:V\times V\rightarrow\mathbb{R}$ by
\[
a(u,v)=\int_{\Omega}\Big( \kappa\nabla u\cdot\nabla v+v\,b\cdot\nabla u \Big).
\]
Then (\ref{eq:original}) can be formulated as: find $u\in V$ such that
\begin{equation}
a(u,v) = (f,v), \quad \forall v\in V
\end{equation}
where $(f,v)$ is the $L^2(\Omega)$ inner product. In order to define our stabilization approach,
we will formulate the above problem following the DPG idea as in \cite{demkowicz2013primal, chan2014robust,
  demkowicz2013robust, demkowicz2014overview, demkowicz2012class,
  zitelli2011class}.
First of all, we assume that a certain trial space $V_{ms}\subset V$ is used.
The trial space consists of local multiscale basis functions $q_i$ with support on local region $\omega_i$ for the $i$-th coarse vertex,
where the coarse neighborhood $\omega_i$ is the union of all coarse elements having the $i$-th vertex, see Figure~\ref{fig:illustration}.
We remark that one can use multiple basis functions per coarse region $\omega_i$.
In addition, we assume that the trial space satisfies
the following stable decomposition property:
\begin{equation}
\label{eq:decomp}
\Big( \sum_{i=1}^{N_c} \|u_i\|_V^2\Big)^{\frac{1}{2}} \leq C_s \|u\|_V, \quad u\in V_{ms}, \; u=\sum_{i=1}^{N_c} u_i, \; \text{supp}(u_i) \subset \omega_i
\end{equation}
where $\|\cdot\|_V$ is a suitable norm which will be defined.
Our main goal is to construct a suitable test space $W_{ms}$,
so that the resulting mixed problem has a good stability property.
For simplicity of our discussion, we assume that there is one basis per node.
Our theory can be generalized to the case that there are multiple trial basis per node.

We will now present the DPG formulation. For each coarse element $K_i \in\mathcal{T}^H$, we define the bilinear forms
\begin{eqnarray}
s^{(i)}(u,v) &=& \int_{K_{i}}\kappa\nabla u\cdot\nabla v \label{eq:s} \\
c^{(i)}(u,v) &=& \int_{K_{i}} \Big( \kappa^{-1} |b|^{2} u \, v+\tilde{\kappa} \, u \, v \Big) \label{eq:c}
\end{eqnarray}
where $\tilde{\kappa} = \kappa \sum_{j=1}^{N_c} |\nabla \chi_j|$ and $\{ \chi_j\}$ is a set of partition of unity functions corresponding to the coarse grid
where the index $j$ denotes $j$-th coarse vertex.
Using the above bilinear forms, we define the inner product
\begin{equation}
(u,v)_{V_i} = s^{(i)}(u,v) + c^{(i)}(u,v)
\end{equation}
with the associated norm $\|u\|_{V_i}^2 = s^{(i)}(u,u) + c^{(i)}(u,u)$ for the space $V(K_i)$,
where we define $V(S) := H^1(S)$ and $V_0(S) := H^1_0(S)$
for a given set $S$.
In addition, we define
$(u,v)_V = \sum_{i=1}^N (u,v)_{V_i}$, which is an inner product for $V$.
We further define
\begin{equation}
a^*(u,v) = a(v,u).
\end{equation}
We next define a linear operator $T: V\rightarrow V$ by
\begin{equation}
a^*(Tu,v) = (u,v)_V, \quad \forall v\in V.
\end{equation}
It is clear that $T$ is a bijective bounded linear operator and hence $T^{-1}$ is bounded.
Using $T^{-1}$, we define the bilinear form $r$ as
\begin{equation}
r(w,z)=(T^{-1}w, T^{-1}z)_{V}.
\end{equation}
Finally, we formulate the problem as:
find $(w,u)\in W_{ms}\times V_{ms}$
such that
\begin{align}
r(w,z)+a(u,z) & =(f,z), \quad \;\forall z\in W_{ms}, \label{eq:mixed1} \\
a^{*}(w,v) & =0,\quad \;\forall v\in V_{ms}. \label{eq:mixed2}
\end{align}

\section{Multiscale stabilization}
\label{sec:stable}

In this section, we will present our stabilization technique.
In Section \ref{sec:test}, we will give the construction of the test functions for a given choice of the trial space $V_{ms}$.
In Section \ref{sec:analysis}, we will give the stability analysis.

\subsection{Construction of test functions}\label{sec:test}

We will present the construction of the test functions. The idea is based on
the CEM-GMsFEM \cite{chung2017constraint}.
We will first define the auxiliary space.
For each coarse element
$K_{i} \in\mathcal{T}^H$, we consider the following eigenvalue problem: find $(\lambda_{j}^{(i)},\phi_{j}^{(i)})\in\mathbb{R}\times V(K_{i})$
such that
\begin{equation}
\label{eq:spectral}
s^{(i)}(\phi^{(i)}_j,v)=\lambda_{j}^{(i)}c^{(i)}(\phi_{j}^{(i)},v), \quad \;\forall v\in V(K_{i}).
\end{equation}
Assume that the eigenvalues are arranged in non-decreasing order, namely
$\lambda_{j}^{(i)}\leq\lambda_{j+1}^{(i)}$. For each $K_i$, we select the first $J_i$ eigenfunctions and define
the local auxiliary space $W_{aux}(K_{i})$ by
\[
W_{aux}(K_{i})=\text{span}\{\phi_{j}^{(i)}|1\leq j\leq J_{i}\}
\]
and the corresponding global auxiliary space $W_{aux}$ by $W_{aux}=\sum_{i}W_{aux}(K_{i})$.
We remark that $J_i$ is the number of small eigenvalues for the problem (\ref{eq:spectral})
and these eigenvalues typically depend on the contrast of the coefficients.

We next present the construction of the multiscale test basis functions. To do so,
we define an operator $\pi:V\rightarrow W_{aux}$ by
\[
\pi(u)=\sum_{i=1}^N\sum_{1\leq j\leq J_{i}}\cfrac{1}{\lambda_{j}^{(i)}}\cfrac{c^{(i)}(u,\phi_{j}^{(i)})}{c^{(i)}(\phi_{j}^{(i)},\phi_{j}^{(i)})} \, \phi_{j}^{(i)}, \quad \;\forall u\in V
\]
where the sum with the index $i$ denotes the sum over all coarse elements $K_i$.
Our multiscale test space $W_{ms}$ consists of two components $W_{ms}^1$ and $W_{ms}^2$. We will define these two spaces as follows,
and will analyze them in Section \ref{sec:analysis}.

Now, we give the definition of the space $W_{ms}^1$.
For a given coarse element $K_i$, we consider its oversampled region $K_i^+$, which is obtained by enlarging $K_i$ by a few coarse grid cells.
Then, for each $\phi^{(i)}_j$ in the auxiliary space $W_{aux}(K_i)$, we consider the following problem: find
 $\psi_{j,ms}^{(i)}\in V_0(K_{i}^{+})$ such that
\[
a^{*}(\psi_{j,ms}^{(i)},v)+c(\pi(\psi_{j,ms}^{(i)}),\pi(v))=c(\pi(\phi_{j}^{(i)}),\pi(v)), \quad \;\forall v\in V_0(K_{i}^{+})
\]
where the bilinear form $c = \sum_{i=1}^N c^{(i)}$.
The first component of test space is defined as
\begin{equation}\label{eq:testsp1}
W_{ms}^{1}=\text{span}\{\psi_{j,ms}^{(i)} \, | \, 1\leq j \leq J_i, \, 1\leq i \leq N \}.
\end{equation}
Next, for each trial basis function $q_{i}\in V_{ms}$ with support $\omega_{i}$ corresponding to the $i$-th coarse vertex,
we will define a local function $\xi_{i}\in V_{0}(\omega_{i})$ such
that
\begin{equation}
\label{eq:xi}
a(v,\xi_{i})=s(q_{i},v), \quad \forall v\in V_{0}(\omega_{i})
\end{equation}
and we define $\eta_{i,ms}\in V(\omega_{i}^{+})$ such that
\[
a^{*}(\eta_{i,ms},v)+c(\pi(\eta_{i,ms}),\pi(v))=(q_{i},v)_V-a(v,\xi_{i}), \quad \;\forall v\in V(\omega_{i}^{+})
\]
where $\omega_i^+$ is obtained by enlarging $\omega_i$ by a few coarse cells, see Figure~\ref{fig:illustration}.
The second component of test space is defined as
\begin{equation}\label{eq:testsp2}
W_{ms}^{2}=\text{span}\{\eta_{i,ms}+\xi_{i} \, | \, 1\leq i \leq N_c \}.
\end{equation}
Finally, our multiscale test space $W_{ms}$ is the sum of $W_{ms}^{1}$ and $W_{ms}^{2}$.

\subsection{Stability analysis}\label{sec:analysis}

We will analyze the stability in this section.
We first notice that the test functions defined in (\ref{eq:testsp1}) and (\ref{eq:testsp2}) have local supports. This is the result
of a localization property of a related space $W_{glo}$, which contains functions defined globally.
The space $W_{glo}$ also consists of two components $W_{glo}^1$ and $W_{glo}^2$.
To define the space $W_{glo}^1$, we find $\psi_j^{(i)} \in V$ such that
\[
a^{*}(\psi_{j}^{(i)},v)+c(\pi(\psi_{j}^{(i)}),v)=c(\phi_{j}^{(i)},\pi(v)), \quad \;\forall v\in V.
\]
Then we define
\begin{equation}\label{eq:testg1}
W_{glo}^{1}=\text{span}\{\psi_{j}^{(i)} \, | \, 1\leq j \leq J_i, \, 1\leq i \leq N \}.
\end{equation}
To define the space $W_{glo}^2$, we find
$\eta_{i}\in V$
such that
\[
a^{*}(\eta_{i},v)+c(\pi(\eta_{i}),\pi(v))=(q_{i},v)_V-a(v,\xi_{i}), \quad\;\forall v\in V
\]
where $\xi_i$ is defined in (\ref{eq:xi}). Then we define
\begin{equation}\label{eq:testg2}
W_{glo}^{2}=\text{span}\{\eta_{i}+\xi_{i} \, | \, 1\leq i \leq N_c \}.
\end{equation}
Finally, we define $W_{glo} = W_{glo}^1 + W_{glo}^2$. Before we discuss our stability results, we will
give a characterization of the space $W_{glo}^1$ in the following lemma.

\begin{lemma}\label{lem:W1}
Let $W_{glo}^1$ be the space defined in (\ref{eq:testg1}). Then, $u\in W_{glo}^1$ if and only if
there is $\phi \in W_{aux}$ such that
\begin{equation}
a^*(u,v) = c(\phi,v), \quad \forall v\in V.
\end{equation}
\end{lemma}
\begin{proof}
First, we will define operators $L_1 : W_{aux} \rightarrow V$ and $L_2 : W_{aux} \rightarrow V$ such that
for a given $v\in W_{aux}$, the images $L_1(v)$ and $L_2(v)$ are defined by solving the following equations
\begin{eqnarray*}
a^*(L_1(v),w) + c(\pi (L_1(v)),\pi(w)) &=&  c(v,w), \quad \forall w\in V, \\
a^*(L_2(v),w) &=&  c(v,w), \quad \forall w\in V.
\end{eqnarray*}
It is clear that $L_1(W_{aux})\subset L_2(W_{aux})$, since for any $v\in W_{aux}$, we have
\[
a^*(L_1(v),w) = c(v-\pi (L_1(v)),w),\quad \forall w\in V,
\]
and therefore $L_1(v)=L_2 (v-\pi (L_1(v)))$.

Next, we define the space $\widehat{W}_{aux}$ as $\widehat{W}_{aux}= \{ v\in W_{aux}|\; c(v,w)= 0 \text{ for all } w\in V \}$. The space $\widetilde{W}_{aux}$ is then defined as the orthogonal complement of $\widehat{W}_{aux}$ with respect to the $c$-inner product such that $W_{aux}=\widehat{W}_{aux}\oplus\widetilde{W}_{aux}$.
Since $W_{aux}$ is a finite dimensional vector space, we have
\[
dim(L_i(W_{aux})) + dim(ker(L_i)) = dim(W_{aux}), \quad i=1,2.
\]
Since $ker(L_i)=\widehat{W}_{aux}$ for $i=1,2$, we have
\[
dim(L_i(W_{aux}))=dim(W_{aux})-dim(\widehat{W}_{aux})=dim(\widetilde{W}_{aux}), \quad i=1,2,
\]
 which implies $dim(L_1(W_{aux})) = dim(L_2(W_{aux}))$.
Hence, we obtain $L_1(W_{aux})=L_2(W_{aux})$. This completes the proof.

\end{proof}

Next, we define a norm
\begin{align*}
\|w\|_{W}  =\sup_{v\in V} \cfrac{a(v,w)}{\|v\|_{V}}.
\end{align*}
\begin{lemma}
For all $w\in V$, we have
\begin{equation}
\|w\|_W=\sqrt{ r(w,w)}.
\end{equation}
\end{lemma}
\begin{proof}
By the definition of $T$, we have
\[
a(v,Tu) = (u,v)_V,\;\forall v\in V.
\]
Thus, by considering $w=Tu$ , we have $a(v,w) = (T^{-1}w,v)_V$.
Therefore, for all $w\in V$, we have
\[
\|T^{-1}w\|^2_V = a(T^{-1}w,w) \leq \|w\|_W \|T^{-1}w\|_V
\]
and
\[
\|w\|_W = \sup_{v\in V} \cfrac{a(v,w)}{\|v\|_{V}} =\sup_{v\in V} \cfrac{(T^{-1}w,v)_V}{\|v\|_{V}} \leq \|T^{-1}w\|_V.
\]
Therefore, we have $\sqrt{r(w,w)}=\|T^{-1}w\|_V=\|w\|_W$.
\end{proof}

Our first result regarding stability is Lemma \ref{lem:glo}.
We consider the problem: find $(w,u)\in W_{glo}\times V_{ms}$
such that
\begin{align*}
r(w,z)+a(u,z) & =(f,z), \quad \;\forall z\in W_{glo},\\
a^{*}(w,v) & =0,\quad \;\forall v\in V_{ms}.
\end{align*}
The following lemma shows that the test space $W_{glo}$ gives perfect stability.

\begin{lemma}\label{lem:glo}
For all $v\in V_{ms}$, there exist a unique $w\in W_{glo}$ such
that
\begin{equation}\label{eq:infsup1}
(u,v)_V=a(u,w), \quad\;\forall u\in V.
\end{equation}
 Therefore we have
\[
\inf_{v\in V_{ms}}\sup_{z\in W_{glo}}\cfrac{a(v,z)}{\|v\|_{V}\|z\|_{W}}=1.
\]
\end{lemma}
\begin{proof}
By definition of $W_{glo}^{2}$, for each $q_i\in V_{ms}$, there exist
a function $\eta_{i}+\xi_{i}\in W_{glo}^{2}$ such that
\[
a^{*}(\eta_{i}+\xi_{i},v)=(q,v)_{V}-c(\pi(\eta_{i}),v), \quad \forall v\in V.
\]
 On the other hand, by Lemma \ref{lem:W1},
 there exist a function $w\in W_{glo}^{1}$ such that
\[
a^{*}(w,v)=c(\pi(\eta_{i}),v), \quad\forall v\in V.
\]
Therefore we have
\[
a^{*}(\eta_{i}+\xi_{i}+w,v)=(q,v)_{V}, \quad\forall v\in V.
\]
This completes the proof for (\ref{eq:infsup1}). To show the second part, we note that (\ref{eq:infsup1})
implies that for every $v\in V_{ms}$, there is $w\in W_{glo}$ such that
\begin{equation*}
\|v\|_V^2 = a(v,w)
\end{equation*}
and that $\|w\|_W = \|v\|_V$. This shows the second part of the lemma.
\end{proof}

We next prove a localization result for our test functions. In particular, we will prove a localization property
for functions in $W_{glo}^1$.
First, we need some notations for the oversampling domain and the cutoff function with respect to these oversampling domains.
For each $K_i$, we recall that $K_{i,m} \subset \Omega$ is the oversampling coarse region by enlarging
$K_{i}$ by $m$ coarse grid layers. For $M>m$, we define $\chi_{i}^{M,m}\in\text{span}\{\chi^{ms}_{j}\}$
such that $0 \leq \chi_i^{M,m} \leq 1$ and
\begin{align}
\chi_{i}^{M,m} & =1\text{ in }K_{i,m}, \label{cutoff1} \\
\chi_{i}^{M,m} & =0\text{ in }\Omega\backslash K_{i,M}. \label{cutoff2}
\end{align}
Note that, we have $K_{i,m} \subset K_{i,M}$. Moreover,
$\chi_i^{M,m}=1$ on the inner region $K_{i,m}$
and $\chi_i^{M,m}=0$ outside the outer region $K_{i,M}$.

\begin{lemma}
\label{lem:decay}
Let $S$ be a given coarse region and let $S_l$ be an oversampling region obtained by enlarging $S$ by $l$ coarse grid layers, where $l \geq 2$.
Let $w_{glo} \in V$ be the solution of
\begin{equation*}
a^{*}(w_{glo},v)+c(\pi(w_{glo}),\pi(v))=L(v), \quad \forall v\in V,
\end{equation*}
where $L(v)$ is a linear functional such that
$L(v)=0,\;\forall v\in V_0(\Omega\backslash S)$. In addition, we let $w_{ms}\in V_0(S_{l})$ be the solution of
\begin{equation*}
a^{*}(w_{ms},v)+c(\pi(w_{ms}),\pi(v))=L(v), \quad\forall v\in V_0(S_{l}).
\end{equation*}
Then, we have
\[
\|w_{glo}-w_{ms}\|_{V}^{2}\leq C\Big(1+\Lambda^{-1}\Big) \Big(1+C^{-1}(1+\Lambda^{-1})^{-1}\Big)^{-(l-1)}
\|w_{glo}\|_{V}^{2},
\]
where $\Lambda = \min_{1\leq i \leq N} \lambda^{(i)}_{J_i+1}$.
\end{lemma}
\begin{proof}
Using the definitions of $w_{glo}$ and $w_{ms}$, we have
\[
a^{*}(w_{glo}-w_{ms},v)+c(\pi(w_{glo}-w_{ms}),\pi(v))=0
\]
 for all $v\in V_0(S_{l})$. Therefore, for all $v\in V_0(S_{l})$, we have
\begin{align*}
\|w_{glo}-w_{ms}\|_{V}^{2} & =\Big(a^{*}(w_{glo}-w_{ms},w_{glo}-w_{ms})+c(\pi(w_{glo}-w_{ms}),\pi(w_{glo}-w_{ms}))\Big)\\
 & =\Big(a^{*}(w_{glo}-w_{ms},w_{glo}-v)+c(\pi(w_{glo}-w_{ms}),\pi(w_{glo}-v))\Big)
\end{align*}
since $b(x)$ is divergence free.
Choosing $v=\chi^{l,l-1}w_{glo}$, we have
\begin{align*}
&\: a^{*}(w_{glo}-w_{ms},(1-\chi_{i}^{l,l-1})w_{glo}) \\
 = &\: \int\kappa\nabla(w_{glo}-w_{ms})\cdot\nabla((1-\chi^{l,l-1})w_{glo})+\int(w_{glo}-w_{ms})b\cdot\nabla((1-\chi^{l,l-1})w_{glo})\\
  \leq &\: \Big(\|w_{glo}-w_{ms}\|_{s(\Omega\backslash S_{l-1})}+\|w_{glo}-w_{ms}\|_{c(\Omega\backslash S_{l-1})}\Big) \, \|(1-\chi^{l,l-1})w_{glo}\|_{s(\Omega\backslash S_{l-1})} \\
  \leq &\: C\Big(\|w_{glo}-w_{ms}\|_{s(\Omega\backslash S_{l-1})}+\|w_{glo}-w_{ms}\|_{c(\Omega\backslash S_{l-1})}\Big) \, \Big(\|w_{glo}\|_{s(\Omega\backslash S_{l-1})}+\|\tilde{\kappa}^{\frac{1}{2}}w_{glo}\|_{L^{2}(\Omega\backslash S_{l-1})}\Big)
\end{align*}
and
\begin{align*}
\|\pi((1-\chi^{(l,l-1)})w_{glo})\|_{c(\Omega\backslash S_{l-1})}^{2} & \leq\|(1-\chi^{(l,l-1)})w_{glo}\|_{c(\Omega\backslash S_{l-1})}^{2}\\
 & =\|w_{glo}\|_{c(\Omega\backslash S_{l})}^{2} + \sum_{K_{i}\subset(S_{l}\backslash S_{i-1})}c^{(i)}((1-\chi^{(l,l-1)})w_{glo},(1-\chi^{(l,l-1)})w_{glo}).
\end{align*}
Next, we will estimate the term $c^{(i)}((1-\chi^{(l,l-1)})w_{glo},(1-\chi^{(l,l-1)})w_{glo}).$
By definition, we have
\begin{align*}
c^{(i)}((1-\chi^{(l,l-1)})w_{glo},(1-\chi^{(l,l-1)})w_{glo}) & =\int_{K_{i}}(\cfrac{|b|^{2}}{\kappa}+\tilde{\kappa})|(1-\chi^{(l,l-1)})w_{glo}|^{2}\\
 & \leq\int_{K_{i}}(\cfrac{|b|^{2}}{\kappa}+\tilde{\kappa})|w_{glo}|^{2}=c^{(i)}(w_{glo},w_{glo}).
\end{align*}
Thus we obtain
\[
\|w_{glo}-w_{ms}\|_{V}^{2}\leq C\Big(\|w_{glo}-w_{ms}\|_{s(\Omega\backslash S_{l-1})}+\|w_{glo}-w_{ms}\|_{c(\Omega\backslash S_{l-1})}\Big) \, \Big(\|w_{glo}\|_{c(\Omega\backslash S_{l-1})}+\|w_{glo}\|_{s(\Omega\backslash S_{l-1})}\Big).
\]
Next we will estimate the terms $\|w_{glo}\|_{c(\Omega\backslash S_{l-1})}$
and $\|w_{glo}-w_{ms}\|_{c(\Omega\backslash S_{l-1})}$.
We will divide the proof in $4$ steps.

{\bf Step 1}:
For a given
$u\in V$ and $K_i\in\mathcal{T}_{H}$, we have
\begin{align*}
\|u\|_{c(K_{i})}^{2} & =\|\pi u\|_{c(K_{i})}^{2}+\|(I-\pi)u\|_{c(K_{i})}^{2}\\
 & \leq\|\pi u\|_{c(K_{i})}^{2}+\cfrac{1}{\Lambda}\|u\|_{s(K_{i})}^{2}\\
 & \leq(1+\Lambda^{-1})\|u\|_{V(K_{i})}^{2}.
\end{align*}
Thus, we have
\begin{equation}
\label{eq:analysis1}
\|w_{glo}-w_{ms}\|_{V}\leq C(1+\Lambda^{-1})^{\frac{1}{2}}\|w_{glo}\|_{V(\Omega\backslash S_{l-1})}.
\end{equation}

{\bf Step 2}:
In this step, we will prove $\|w_{glo}\|_{V(\Omega\backslash S_{k})}\leq C\|w_{glo}\|_{V(S_{k}\backslash S_{k-1})}$.
By direct computations, we have
\[
\int_{\Omega}\kappa\nabla w_{glo}\cdot\nabla(1-\chi_{i})w_{glo}=\|w_{glo}\|_{s(\Omega\backslash S_{k})}^{2}+\int_{S_{k}\backslash S_{k-1}}\kappa\nabla w_{glo}\cdot\nabla(1-\chi_{i})w_{glo}
\]
and
\begin{align*}
\int_{\Omega} b\cdot\nabla w_{glo}(1-\chi_{i})w_{glo} & =\cfrac{1}{2}\int_{\Omega\backslash S_{l-1}}(b\cdot\nabla w_{glo})(1-\chi_{i})w_{glo}-\cfrac{1}{2}\int_{\Omega\backslash S_{l-1}} b\cdot\nabla((1-\chi_{i})w_{glo})w_{glo}
\end{align*}
and
\begin{align*}
c(\pi w_{glo},\pi((1-\chi_{i})w_{glo})) & =\|\pi(w_{glo})\|_{c(\Omega\backslash S_{k-1})}^{2}+\sum_{K_{i}\subset S_{k}\backslash S_{k-1}}c^{(i)}(\pi w_{glo},\pi\Big((1-\chi^{(k,k-1)})w_{glo}\Big)).
\end{align*}
Using the above equations, we have
\begin{eqnarray*}
\|w_{glo}\|_{V(\Omega\backslash S_{k})}^{2} & = & -\int_{S_{k}\backslash S_{k-1}}\kappa\nabla w_{glo}\cdot\nabla(1-\chi_{i})w_{glo}\\
 &  & -\cfrac{1}{2}\Big(\int_{S_{l}\backslash S_{l-1}} b\cdot\nabla w_{glo}(1-\chi^{(k,k-1)})w_{glo}-\int_{S_{l}\backslash S_{l-1}} b\cdot\nabla\Big((1-\chi^{(k,k-1)})w_{glo}\Big)w_{glo}\Big)\\
 &  & -\sum_{K_{i}\subset S_{k}\backslash S_{k-1}}c^{(i)}(\pi w_{glo},\pi\Big((1-\chi^{(k,k-1)})w_{glo}\Big)) \\
 &=& T_1 + T_2 + T_3.
\end{eqnarray*}
Next, we will estimate the term $T_1$.
Clearly,
\begin{align*}
\int_{S_{k}\backslash S_{k-1}}\kappa\nabla w_{glo}\cdot\nabla(1-\chi^{(k,k-1)})w_{glo} & \leq C\|w_{glo}\|_{s(S_{k}\backslash S_{k-1})} \Big(\|w_{glo}\|_{s(S_{k}\backslash S_{k-1})}+\|w_{glo}\|_{c(S_{k}\backslash S_{k-1})}\Big)\\
 & \leq C(1+\Lambda^{-1})^{\frac{1}{2}}\|w_{glo}\|_{V(S_{k}\backslash S_{k-1})}\|w_{glo}\|_{s(S_{k}\backslash S_{k-1})}.
\end{align*}
Secondly, we will estimate the term $T_2$. We have
\begin{align*}
 & -\cfrac{1}{2}\Big(\int_{S_{l}\backslash S_{l-1}}(b\cdot\nabla w_{glo})(1-\chi^{(k,k-1)})w_{glo}-\int_{S_{l}\backslash S_{l-1}} b\cdot\nabla\Big((1-\chi^{(k,k-1)})w_{glo}\Big)w_{glo}\Big)\\
\leq & C\|w_{glo}\|_{s(S_{k}\backslash S_{k-1})}\|\cfrac{|b|}{\kappa^{\frac{1}{2}}}w_{glo}\|_{L^{2}(S_{k}\backslash S_{k-1})}+\|(\cfrac{|b|^{2}}{\kappa})^{\frac{1}{2}}w_{glo}\|_{L^{2}(S_{k}\backslash S_{k-1})}\|\tilde{\kappa}^{\frac{1}{2}}w_{glo}\|_{L^{2}(S_{k}\backslash S_{k-1})}\\
\leq & C\Big(\|w_{glo}\|_{s(S_{k}\backslash S_{k-1})}+\|w_{glo}\|_{s(S_{k}\backslash S_{k-1})}\Big)\|w_{glo}\|_{c(S_{k}\backslash S_{k-1})}.
\end{align*}
Finally, we will estimate the term $T_3$. We have
\begin{align*}
-\sum_{K_{i}\subset S_{k}\backslash S_{k-1}}c^{(i)}(\pi w_{glo},\pi\Big((1-\chi^{(k,k-1)})w_{glo}\Big)) & \leq\|\pi w_{glo}\|_{c(S_{k}\backslash S_{k-1})}\|(1-\chi^{(k,k-1)})w_{glo}\|_{c(S_{k}\backslash S_{k-1})}\\
 & \leq C(1+\Lambda^{-1})^{\frac{1}{2}}\|\pi w_{glo}\|_{c(S_{k}\backslash S_{k-1})}\|w_{glo}\|_{V(S_{k}\backslash S_{k-1})}.
\end{align*}
Combining the above results, we have
\begin{equation}
\label{eq:analysis2}
\|w_{glo}\|_{V(\Omega\backslash S_{k})}^{2}\leq C(1+\Lambda^{-1})\|w_{glo}\|_{V(S_{k}\backslash S_{k-1})}^{2}.
\end{equation}

{\bf Step 3}: In this step, we will prove that
$\|w_{glo}\|_{V(\Omega\backslash S_{k})}\leq(1+C^{-1}(1+\Lambda^{-1})^{-1})^{-1}\|w_{glo}\|_{V(S_{k}\backslash S_{k-1})}.$
Indeed, we have
\begin{equation}
\label{eq:analysis3}
\begin{split}
\|w_{glo}\|_{V(\Omega\backslash S_{k-1})}^{2} & =\|w_{glo}\|_{V(\Omega\backslash S_{k})}^{2}+\|w_{glo}\|_{V(S_{k}\backslash S_{k-1})}^{2}\\
 & \geq(1+C^{-1}(1+\Lambda^{-1})^{-1})\|w_{glo}\|_{V(\Omega\backslash S_{k})}^{2}
\end{split}
\end{equation}

{\bf Step 4}: In this step, we will prove the required estimate.
Using (\ref{eq:analysis1}), (\ref{eq:analysis2}) and (\ref{eq:analysis3}), we have
\begin{align*}
\|w_{glo}-w_{ms}\|_{V}^{2} & \leq C(1+\Lambda^{-1})\|w_{glo}\|_{V(\Omega\backslash S_{l-1})}^{2}\\
 & \leq C(1+\Lambda^{-1})(1+C^{-1}(1+\Lambda^{-1})^{-1})^{-(l-1)}\|w_{glo}\|_{V(\Omega\backslash S)}^{2}\\
 & \leq C(1+\Lambda^{-1})(1+C^{-1}(1+\Lambda^{-1})^{-1})^{-(l-1)}\|w_{glo}\|_{V}^{2}.
\end{align*}
\end{proof}

The following is the main result of this section. It states that our test space gives a stable numerical scheme.

\begin{theorem}\label{thm:stab}
Assume that $N_{d}^{\frac{1}{2}}CED < 1$.
For any given $u\in V_{ms}$, there exists a function $w\in W_{ms}$ such that
\[
\cfrac{1-N_{d}^{\frac{1}{2}}CED}{1+N_{d}^{\frac{1}{2}}CED}\, \|u\|_{V}\leq\cfrac{a(u,w)}{\|w\|_{W}}
\]
where $E = C\Big(1+\Lambda^{-1}\Big) \Big(1+C^{-1}(1+\Lambda^{-1})^{-1}\Big)^{-(l-1)}$ is the same constant in Lemma~\ref{lem:decay},
$N_d$ is the maximum of numbers of coarse grid vertices and cells and
$D = \kappa_0^{-1} \max\{ \kappa^{-1} |b|^2, \tilde{\kappa} \}$.
\end{theorem}
\begin{proof}
Since $u\in V_{ms}$,
we can write $u=\sum_{i=1}^{N_c} d_{i}q_{i}$ as a linear combination of trial basis functions $\{q_i\}$ in $V_{ms}$,
where $d_i$ are scalars.
By Lemma \ref{lem:glo}, there exists a function $w_{glo}\in W_{glo}$
such that
\[
(u,v)_{V}=a(v,w_{glo}), \quad\forall v\in V
\]
and we can write
\[
w_{glo}=w_{1}+w_{2}
\]
 where
\begin{align*}
a^{*}(w_{1},v)+c(\pi w_{1},v) & =(u,v)_{V}, \quad\forall v\in V, \\
a^{*}(w_{2},v) & =-c(\pi w_{1},v), \quad\forall v\in V.
\end{align*}
Next, we define $w_{1}^{(i)}\in V,\; w_{2}^{(j)}\in V$ such that
\begin{align*}
a^{*}(w_{1}^{(i)},v)+c(\pi w_{1}^{(i)},v) & =(d_{i}q_{i},v)_{V}, \quad\forall v\in V, \\
a^{*}(w_{2}^{(j)},v)+c(\pi w_{2}^{(j)},v) & =c^{(j)}((\pi w_{2}-\pi w_{1})|_{K_{j}},v), \quad\forall v\in V
\end{align*}
and $w_{1,ms}^{(i)}\in W_{ms},\; w_{2,ms}^{(j)}\in W_{ms}$ be the corresponding localized functions defined by
\begin{align*}
a^{*}(w_{1,ms}^{(i)},v)+c(\pi w_{1,ms}^{(i)},v) & =(d_{i}q_{i},v)_{V}, \quad\forall v\in V_{0}(\omega_{i}^{+}), \\
a^{*}(w_{2,ms}^{(j)},v)+c(\pi w_{2,ms}^{(j)},v) & =c^{(j)}((\pi w_{2}-\pi w_{1})|_{K_{j}},v), \quad\forall v\in V_{0}(K_{j}^{+}).
\end{align*}
Clearly, we have $w_{glo}=\sum_{i}w_{1}^{(i)}+\sum_{j}w_{2}^{(j)}$,
where the index $i$ corresponds to coarse vertices and the index $j$ corresponds to coarse cells.
We take $w=\sum_{i}w_{1,ms}^{(i)}+\sum_{j}w_{2,ms}^{(j)}\in W_{ms}$. Then we have the following
\begin{equation}
\label{eq:conv1}
\begin{split}
\|u\|_{V}^{2} & =a(u,w_{glo})=a(u,w_{glo}-w)+a(u,w)\\
 & =\sum_i a(u,w_{1}^{(i)}-w_{1,ms}^{(i)})+\sum_j a(u,w_{2}^{(j)}-w_{2,ms}^{(j)})+a(u,w).
\end{split}
\end{equation}
Notice that, using Lemma~\ref{lem:decay},
the first two terms on the right hand side of (\ref{eq:conv1}) can be estimated as follows
\begin{align*}
a(u,w_{1}^{(i)}-w_{1,ms}^{(i)}) & \leq C\|u\|_{V}\, \|w_{1}^{(i)}-w_{1,ms}^{(i)}\|_{V} \\
 & \leq CE\|u\|_{V} \, \|d_{i}q_{i}\|_{V}
\end{align*}
and
\begin{align*}
a(u,w_{2}^{(j)}-w_{2,ms}^{(j)}) & \leq C\|u\|_{V} \, \|w_{2}^{(j)}-w_{2,ms}^{(j)}\|_{V} \\
 & \leq CE\|u\|_{V} \, \|(\pi w_{2}-\pi w_{1})|_{K_{j}}\|_{c}
\end{align*}
where the $c$-norm is defined as $\|w\|_c^2 = c(w,w)$.
Therefore, the first two terms on the right hand side of (\ref{eq:conv1}) can be estimated as
\begin{align*}
\sum_i a(u,w_{1}^{(i)}-w_{1,ms}^{(i)})+\sum_j a(u,w_{2}^{(j)}-w_{2,ms}^{(j)}) & \leq  N_{d}^{\frac{1}{2}}CE\|u\|_{V}\, \Big((\sum_i d_{i}^{2}\|q_{i}\|_{V}^{2})^{\frac{1}{2}}+\|(\pi w_{2}-\pi w_{1})\|_{c}\Big).
\end{align*}
Notice that, by the Poincare inequality, we have
\begin{equation*}
\|w_{2}\|_{c} \leq C D \|w_2\|_s.
\end{equation*}
By the definition of $w_2$, we have $\|w_2\|_s \leq C \|w_1\|_c$, and by the definition of $w_1$, we have $\|w_1\|_c \leq C \|u\|_V$.
Furthermore, by the assumption on stable decomposition (\ref{eq:decomp}), we have $(\sum d_{i}^{2}\|q_{i}\|_{V}^{2})^{\frac{1}{2}}\leq C_s\|\sum_i d_{i}q_{i}\|_V$.
Thus we have
\begin{align*}
\sum_i a(u,w_{1}^{(i)}-w_{1,ms}^{(i)})+\sum_j a(u,w_{2}^{(j)}-w_{2,ms}^{(j)}) & \leq N_{d}^{\frac{1}{2}} CE \|u\|_V \, \Big(\|\sum_i d_{i}q_{i}\|_V+\|\pi w_{2}\|_c+\|\pi w_{1}\|_{c} \Big)\\
 & \leq N_{d}^{\frac{1}{2}} CE D \|u\|_V^2.
 \end{align*}
Similarly, for all $v\in V_{ms}$, we have
\[
a(v,w_{glo}-w)\leq N_{d}^{\frac{1}{2}} CED \|u\|_V \, \|v\|_V.
\]
Finally, we have
\begin{align*}
\cfrac{a(u,w)}{\|w\|_{W}} & \geq\cfrac{a(u,w-w_{glo})+a(u,u_{glo})}{\|w_{glo}\|_{W}+\|w-w_{glo}\|_{W}}\\
 & \geq\cfrac{1-N_{d}^{\frac{1}{2}}CED}{1+N_{d}^{\frac{1}{2}}CED}\|u\|_V
\end{align*}
for $N_{d}^{\frac{1}{2}}CED<1$. This completes the proof.
\end{proof}

In the above theorem, we assume that $N_{d}^{\frac{1}{2}}CED<1$.
This can be achieved by using large enough number of layers $l$ in the construction of oversampling layers
in the definitions of test spaces (\ref{eq:testsp1}) and (\ref{eq:testsp2}).

\section{Numerical results}
\label{sec:numer}

In this section, we will present some numerical examples to demonstrate the performance of the method. For the following example, we consider $h=1/200$ and $\Omega=[0,1]^2$.
We will show the performance by considering various coarse grid sizes and number of oversampling layers.
We notice that the number of test functions in the space $W_{ms}^1$ depends on the number of eigenfunctions selected in the auxiliary spectral problem. Thus,
we will consider various choices of this number and show that one needs to include enough eigenfunctions to obtain stability.
For the space $W_{ms}^2$, its dimension depends on the number of trial basis functions. In our simulations, we choose piecewise linear functions as our trial basis.

We next discuss some implementation details. We use $A$ to denote a fine scale discretization of the original problem (\ref{eq:original}).
We use the matrix $Q$ to represent the matrix representation of the trial basis functions, and the matrix $W$ to represent the matrix representation
of the test functions. In addition, we use the matrix $V$ to denote the matrix representation of the inner product $(\cdot,\cdot)_V$. Then
the matrix form of (\ref{eq:mixed1})-(\ref{eq:mixed2}) is given by
\begin{equation}
\label{eq:discrete}
\begin{split}
W^T A V^{-1} A^T W \vec{w} + W^T A Q \vec{u} &= W^T F \\
Q^T A^t W \vec{w} &= 0
\end{split}
\end{equation}
where $\vec{u}$ and $\vec{w}$ denote the vector representations of $u$ and $w$, and $F$ is the vector representation of $f$.
We observe that (\ref{eq:discrete}) contains the matrix $V^{-1}$.
In simulations, we will replace $V^{-1}$ by the matrix $B^{-1}$ where $B$ is the matrix representation of the inner product $c(\cdot,\cdot)$,
where $c$ is defined in (\ref{eq:c}), and can be diagonalized by mass lumping.
The motivation of this replacement is that the norm induced by $s^{(i)}$, defined in (\ref{eq:s}),
can be controlled by the norm induced by $c^{(i)}$ in the discrete case under an assumption. Notice that, by the inverse inequality, we have
\begin{equation*}
s^{(i)}(u,u) \leq C h^{-2} \Big(\max_{x\in K_i} \frac{\kappa(x)}{|b(x)|} \Big)^2 c^{(i)}(u,u),
\end{equation*}
where $h$ is the fine mesh size.
Thus, if we assume the fine mesh size satisfies $h^{-1} \Big(\max_{x\in K_i} \frac{\kappa(x)}{|b(x)|} \Big) = O(1)$,
then we have the desired replacement.

\subsection{Example 1}
In our first example, we consider a constant diffusion coefficient, that is, $\kappa=1/200$. The velocity field
is given by $b  = (\cos(18\pi y)\sin(18\pi x) , -\cos(18\pi x)\sin(18\pi y)  )$ and the source term is given by $f=1$.
The numerical results are shown in Table~\ref{table:ex1}.
In the first column, we present the number of test functions used for the space $W_{ms}^1$ per coarse element.
The second column shows the coarse mesh size,
and the third column shows the number of oversampling layers used in the constructions of test functions for both $W_{ms}^1$ and $W_{ms}^2$.
Finally, in the last column, we present the relative errors in the $V$-norm.
For comparison purpose, we show the projection errors in $V$-norm in parenthesis,
where the projection error is obtained by projecting the true solution in the trial space using the $V$-inner product.
From the results in Table~\ref{table:ex1}, we observe that the error is close to the projection error once sufficient oversampling layers are used
in the construction of test functions.

\begin{table}[!ht]
\centering
\begin{tabular}{|c|c|c|c|}
\hline
\#basis($W^1_{ms}$) & $H$ & \#layer & $V$-norm (projection error)\tabularnewline
\hline
3 & 1/10& 3 & 3.05\%(2.92\%)\tabularnewline
\hline
3 & 1/20& 4 &2.19\%(2.18\%)\tabularnewline
\hline
3 & 1/40& 5 &0.85\%(0.85\%)\tabularnewline
\hline
\end{tabular}
\caption{Numerical results for Example 1.}
\label{table:ex1}
\end{table}

\subsection{Example 2}
In our second example, we perform a similar test as in Example 1, but we use $\kappa=1/2000$
and $b=(-\partial_{y}H,+\partial xH)$, where $H=(sin(5\pi x)sin(6\pi y)/(60\pi))+0.005(x+y)$. In this case,
the Peclet number is larger than that of Example 1.
The numerical results are shown in Table~\ref{table:ex2}.
We observe similar performance as in Example 1.

\begin{table}[!ht]
\centering
\begin{tabular}{|c|c|c|c|}
\hline
\#basis($W^1_{ms}$) & $H$ & \#layer & $V$-norm (projection error)\tabularnewline
\hline
3 & 1/10& 3 & 11.79\%(11.07\%)\tabularnewline
\hline
3 & 1/20& 4 &3.25\%(3.24\%)\tabularnewline
\hline
3 & 1/40& 5 &0.69\%(0.68\%)\tabularnewline
\hline
\end{tabular}
\caption{Numerical results for Example 2.}
\label{table:ex2}
\end{table}

\subsection{Example 3}
Finally, we consider a heterogenous velocity field defined by a Darcy flow
in a high contrast medium. In particular, the velocity $b$ is defined by the following system
\begin{align*}
K^{-1}b & =-\nabla p\\
\nabla\cdot b & =q
\end{align*}
where
\[
q(x)=\begin{cases}
1 & x\in[0,\cfrac{1}{10}]\times[0,\cfrac{1}{10}]\\
-1 & x\in[\cfrac{9}{10},1]\times[\cfrac{9}{10},1]\\
0 & \text{otherwise}
\end{cases}
\]
and
\[
f(x)=\begin{cases}
1 & x\in[0,\cfrac{1}{10}]\times[0,\cfrac{1}{10}]\\
0 & \text{otherwise}
\end{cases}
\]
and the coefficient $K$ is shown in Figure~\ref{fig:kappa}, where the contrast is $10^4$. In addition, we take $\kappa=1/20$.
The numerical results are presented in Table~\ref{table:ex3}.
We observe that the solution error is very close to the projection error once a sufficient number of oversampling layers is used
in the construction of test functions.
This result confirm that our test space gives very good stability, even for high Peclet numbers.

\begin{figure}[!ht]
\centering
\includegraphics[scale=0.4]{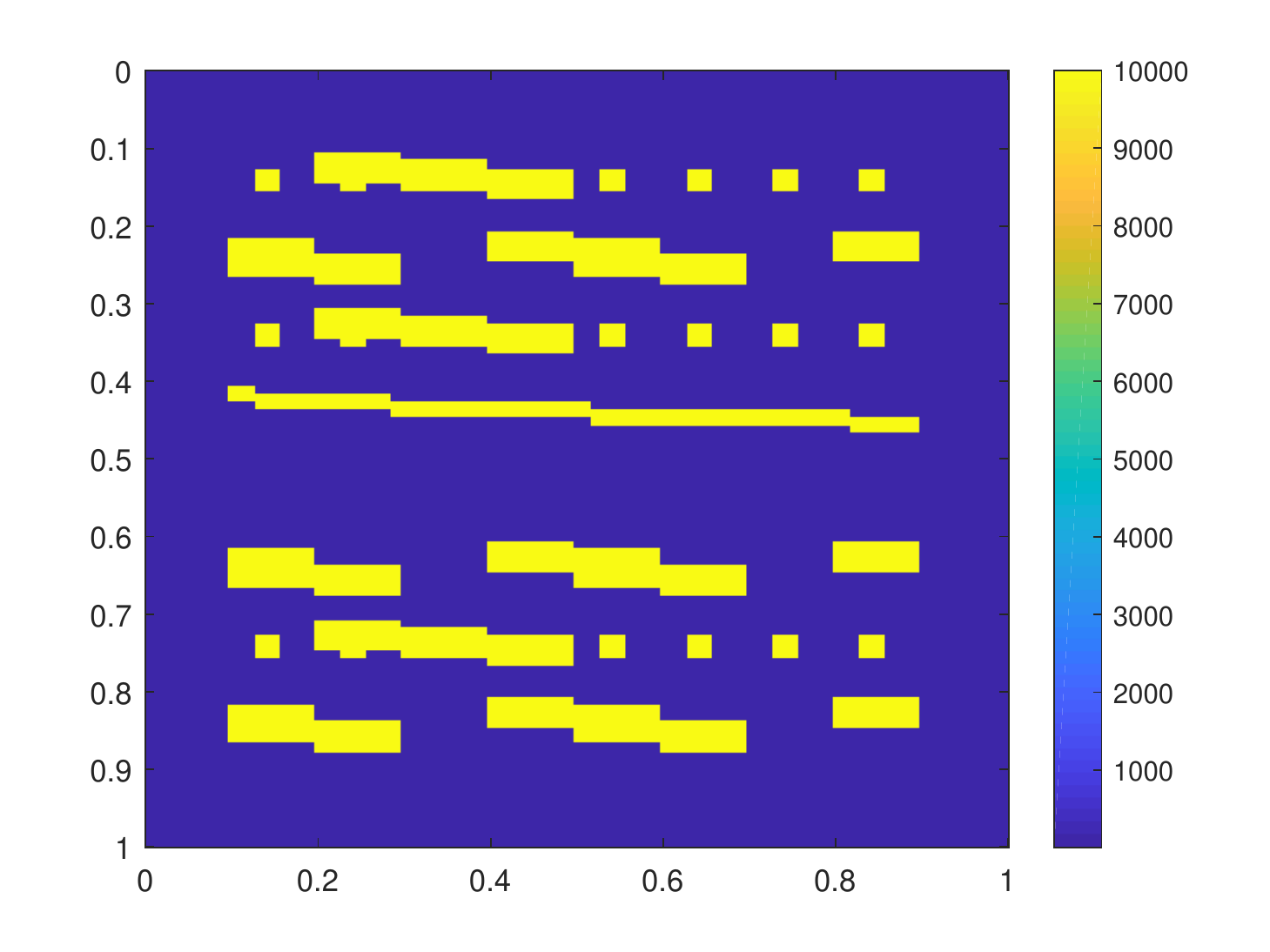}

\caption{The coefficient $K$ for Example 3.}
\label{fig:kappa}
\end{figure}

\begin{table}[!ht]
\centering
\begin{tabular}{|c|c|c|c|}
\hline
\#basis( $W^1_{ms}$) & $H$ & \#layer & $V$-norm (projection error)\tabularnewline
\hline
3 & 1/10& 3 & 16.16\%(15.86\%)\tabularnewline
\hline
3 & 1/20& 4 &4.61\%(4.59\%)\tabularnewline
\hline
3 & 1/40& 5 &1.20\%(1.20\%)\tabularnewline
\hline
\end{tabular}
\caption{Numerical results for Example 3.}
\label{table:ex3}
\end{table}

\section{Conclusion}
We have presented a new stabilization technique for multiscale convection diffusion problems.
The proposed methodology is based on the DPG idea with a suitable choice of test functions.
The construction of the test function is based on the CEM-GMsFEM approach. We show that,
once a sufficient number of oversampling layers is used, the resulting test functions have a decay property,
and give a good stability. We also present numerical results to confirm this theory.

\bibliographystyle{plain}
\bibliography{references}

\end{document}